\documentclass[12pt]{article}

\usepackage{amsmath,amsthm,amssymb}

\usepackage[cp1251]{inputenc}
\usepackage[english]{babel}

\usepackage{authblk}

\usepackage{cite}
\usepackage[colorlinks,linkcolor=blue,citecolor=blue,filecolor=blue,urlcolor=blue]{hyperref}
\usepackage[ruled,vlined,linesnumbered]{algorithm2e}

\newtheorem{theorem}{Theorem}

\newtheorem{lemma}[theorem]{Lemma}

\newtheorem{example}[theorem]{Example}

\begin{document}

\title{Residue Integrals and Waring's Formulas for a Class of Systems of Transcendental Equations in $\mathbb{C}^n$}

\author[1]{A.~A.~Kytmanov \thanks{E-mail: \href{mailto:aakytm@gmail.com}{aakytm@gmail.com}}}
\author[2]{A.~M.~Kytmanov \thanks{E-mail: \href{mailto:akytmanov@sfu-kras.ru}{akytmanov@sfu-kras.ru}}}
\author[2]{E.~K.~Myshkina \thanks{E-mail: \href{mailto:elfifenok@mail.ru}{elfifenok@mail.ru}}}

\affil[1]{Institute of Space and Information Technology, Siberian Federal University, 79 Svobodny av., 660041 Krasnoyarsk, Russia}
\affil[2]{Institute of Mathematics and Computer Science, Siberian Federal University, 79 Svobodny av., 660041 Krasnoyarsk, Russia}

\date{}

\maketitle

\begin{abstract}
The present article is focused on the study of a special class of systems of nonlinear transcendental equations for which classical algebraic and symbolic methods are inapplicable. For the purpose of the study of such systems, we develop a method for computing residue integrals with integration over certain cycles. We describe conditions under which the mentioned residue integrals coincide with power sums of the inverses to the roots of a system of equations (i.e., multidimensional Waring's formulas). As an application of the suggested method, we consider a problem of finding sums of multi-variable number series.

\medskip


\medskip

\end{abstract}

\section{Introduction}

The problem of elimination of unknowns from systems of nonlinear algebraic equations is a classical algebraic problem. Its solution based on the notion of resultant was developed in works of Silvester and B\'ezout. This method was described in detail in the classical Van~der~Waerden's monograph \cite{Wae50}. In the middle of the 20th century, B.~Buchberger suggested a new elimination method based on the notion of a Gr\"obner basis. Nowadays it is one of the main elimination methods in polynomial computer algebra (see, e.g., \cite{Buc85,AL94}).

In the 1970s in \cite{Aiz77} L.~A.~Aizenberg proposed a new elimination method based on the multidimensional residue theory, namely on the formulas of multidimensional logarithmic residue and Grothendieck residue. The basic idea of the method was to find certain residue integrals connected to the power sums of roots of a given system of equations without finding the roots themselves. (The formulas for computing power sums may vary depending on the given system of equations.) Then, using the classical recurrent Newton formulas, one can construct a polynomial whose roots coincide with the first coordinates of the roots of the given system with the same multiplicity (i.e., resultant). This method does not increase the multiplicity of roots in comparison with the classical method (see, e.g., \cite{Wae50}). Its further developments were implemented in \cite{AY83}, \cite{Ts92}, and \cite{BKL98}. 

In applied problems of chemical kinetics systems of transcendental equations, namely, systems consisting of exponential polynomials \cite{Byk06,BTs11} (Zeldovich--Semenov model, etc.) arise as well. However, the elimination method developed in \cite{AY83,Ts92,BKL98} cannot be applied to this kind of systems. One of the obstacles is the fact that the set of roots of a system of $n$ transcendental equations in $n$ variables is, in general, infinite. Moreover, multi-Newton sums (with powers in $\mathbb N^n$) of the roots of such systems lead usually to divergent series. At the same time, the multidimensional logarithmic residue and Grothendieck residue formulas are not applicable to the residue integrals that arise in such kind of systems which means that one is unable to calculate power sums of the roots. Therefore, the known methods have to be modified significantly in order to be applied to such systems. In particular, one has to be able to compute power sums of the inverses to the roots (without finding the roots themselves). Then, using the obtained formulas together with analogs of recurrent Newton formulas and Waring's formulas for entire functions of a complex variable (see \cite[Ch.~1]{BKL98}), one can construct resultant, which is also an entire function. But nevertheless, the formulas for finding power sums are still the main component of the method.

Classical Waring's formulas express power sums of the roots of a polynomial in terms of its coefficients (see, e.g., \cite{Wae50}). Multidimensional analogs of Waring's formulas for certain types of algebraic systems were developed in \cite{KS15}.

In the works \cite{KP05,BKM07,AKyt09,KM13,KKM15,KM15,KM16,KKh16} simple enough classes of systems of equations containing entire or meromorphic functions were considered. An algorithm that computes the residue integrals and applies to them the recurrent Newton formulas was given in \cite{AKyt10}. In \cite{Kho14} the developed methods were applied to study of a system (consisting of exponential polynomials) that arises in Zeldovich--Semenov model.

We now present a more precise review of the known results. In \cite{KP05} the authors considered the system of functions:
\[
f_1(z),\ldots, f_n(z),
\]
where $z=(z_1,\ldots, z_n)$. Each $f_j(z)$ is analytic in the neighborhood of $0\in\mathbb C^n$ and is defined by
\[
f_j(z)=z^{\beta^j}+Q_j(z),\quad j=1,\ldots,n,
\]
where $\beta^j =(\beta_1^j,\ldots,\beta_n^j)$ is a multi-index with integer nonnegative coordinates, $z^{\beta^j}=z_1^{\beta_1^j}\cdot \ldots \cdot z_n^{\beta_n^j}$, and $\|\beta^j\|=\beta_1^j+\ldots +\beta_n^j=k_j$, $j=1,\ldots,n$. Functions $Q_j$ are expanded in a neighborhood of origin into an absolutely and uniformly converging Taylor series of the form
\[
Q_j(z)=\sum_{\|\alpha\|> k_j}a_\alpha^j z^\alpha,
\]
where $\alpha=(\alpha_1,\ldots,\alpha_n)$, $\alpha_j\geqslant 0$, $\alpha_j\in\mathbb Z$, and $z^\alpha=z_1^{\alpha_1}\cdot \ldots \cdot z_n^{\alpha_n}$.

The formulas for computing residue integrals
\[
J_\beta=\frac 1{(2\pi i)^n}\int\limits_{\gamma(r)}\frac 1{z^{\beta+U}}\cdot\frac{df}{f}
\]
in terms of the coefficients of $Q_j(z)$ were obtained. Here $\gamma(r)=\{z=(z_1,\ldots, z_n): |z_j|=r_j, \ j=1,\ldots, n\}$ and $U=(1,\ldots, 1)$ is the unit vector.

Multidimensional Newton formulas for such systems were obtained in \cite{AKyt09} and \cite{AKyt10}.

A class of systems containing the functions
\begin{equation}\label{f1}
f_j(z)=\bigl(z^{\beta^j}+Q_j(z)\bigr)e^{P_j},\quad j=1,\ldots,n.
\end{equation}
was considered in \cite{KM13} and \cite{KKM15}. A method for finding residue integrals for such systems was given in \cite{KKM15}.

In the present work we compute residue integrals for a specific kind of systems of $n$ transcendental equations, and deduce from this computation (provided such series converge) the values of the sums of series (with powers in $(-\mathbb{N})^n$) consisting of the roots of such systems which do not belong to coordinate subspaces. In other words, we generalize the statements from \cite{KP05,BKM07,AKyt09,KM13,KKM15} to a wider class of systems of transcendental equations, where instead of the monomials $z^{\beta^j}$ in \eqref{f1} we consider products of linear functions. Our objectives are to obtain formulas for computing residue integrals, to study the connection between residue integrals and power sums of the inverses to the roots (Waring's formulas), and to introduce a scheme for elimination of unknowns from the considered class of systems.

\section{Calculation of residue integrals}

In this section, we introduce the class of systems of transcendental equations that will be considered in this work. A.~Tsikh considered its algebraic analog in \cite{Ts80} (see also \cite[Theorems 8.5, 8.6]{BKL98}) and studied the number of its common roots in $\overline{\mathbb{C}}^n$. Theorem~\ref{t1} shows that for any such system the residue integral $J_\gamma(t)$ (where $t>0$ is sufficiently small) can be computed by means of converging series of Taylor coefficients of the functions contained in the initial system.

For $z=(z_1,\ldots,z_n)\in\mathbb{C}^n$ consider a system of functions 
\begin{equation}\label{f2}
\begin{cases}
f_1(z)=q_1(z)+Q_1(z),
\\ \ldots \\
f_n(z)=q_n(z)+Q_n(z),
\end{cases}
\end{equation}
where $Q_i(z)$ are entire functions and 
\begin{equation}\label{f3}
q_i(z_1, \ldots, z_n)=(1-a_{i1}z_1)^{m_{i1}}\cdot \ldots \cdot (1-a_{in}z_n)^{m_{in}}
\end{equation}
for $i=1,\ldots, n$. Here $m_{ij}$ are positive integers and $a_{ij}$ are complex numbers, such that $a_{ij}\neq a_{kj}, \ i\neq k$.

Let $J=(j_1,\ldots,j_n)$ be a multi-index where $(j_1 \ldots j_n)$ is a permutation of $(1 \ldots n)$. Then by $a_J$ we denote the vector $(a_{1j_1},\ldots, a_{nj_n})$.
For each $i$ we define a function
\[
h_i(z)=\begin{cases}
q_i(z), & \mbox{if}\ a_{ij}\neq 0\ \mbox{for each}\ j; \\
q_i(z)\cdot \frac{1}{z_{j_1}}\cdot \ldots \cdot \frac{1}{z_{j_k}}, & \mbox{if}\ a_{i{j_1}}=\ldots =a_{i{j_k}}=0.
\end{cases}
\]

The system
\begin{equation}\label{f4}
h_i(z)=0, \quad i=1,\ldots, n
\end{equation}
has $n!$ isolated roots in $\overline{\mathbb C}^n$, where $\overline{\mathbb C}^n=\overline{\mathbb C}\times\ldots \times \overline{\mathbb C}$. Since $\overline{\mathbb{C}}$ is a compactification of the complex plane $\mathbb{C}$, then $\overline{\mathbb{C}}^n$ is one of the known compactifications of $\mathbb{C}^n$. The roots of \eqref{f4} are
\[
\tilde a_J=\begin{cases}
\bigl(1/a_{1j_{1}},\ldots, 1/a_{nj_{n}}\bigr), & \text{if}\ a_{kj_{k}}\neq 0\ \text{for each}\ k=1,\ldots, n, \\
\bigl(1/a_{1j_{1}},\ldots,\infty_{[i_1]}, \ldots,\infty_{[i_k]}, \ldots, 1/a_{nj_{n}}\bigr), & \text{if}\ a_{i_1{j_{i_1}}}=\ldots =a_{i_k{j_{i_k}}}=0,
\end{cases}
\]
where $k,j=1,\ldots,n$, and $J=(j_1,\ldots,j_n)$. Note that we write $\infty$ (as a point in $\overline{\mathbb{C}}$) in $\tilde a_J$ whenever $a_{kj_k}=0$.

By $\Gamma_h$ we denote the cycle
\begin{equation}\label{f5}
\Gamma_h=\{z\in\mathbb C^n\colon |h_i(z)|=r_i,\ r_i>0,\ i=1,\ldots,n\}.
\end{equation}

Now we define the cycle $\Gamma_{h,\tilde a_J}$ by
\begin{equation}\label{f6}
\begin{cases}
|l_1|=r_1, \\ \ldots \\ |l_n|=r_n,
\end{cases}
\quad \text{where} \quad 
\begin{cases}
l_k=1-a_{kj_{k}}z_k, & \text{if} \ a_{k j_k}\neq 0, \\ l_k=1/z_k, & \text{if} \ a_{k j_k}=0.
\end{cases}
\end{equation}

\begin{lemma}\label{l1}
For sufficiently small $r_i$ a global cycle $\Gamma_h$ defined by \eqref{f5} has connected components (local cycles) in the neighborhoods of the roots $a_J$. Moreover, $\Gamma_h$ is homologous to the sum of the local cycles $\Gamma_{h,\tilde a_J}$.
\end{lemma}
\begin{proof}
Consider the global cycle $\Gamma_h$ defined by
\[
\begin{cases}
\left|h_1\right|=r_1, \\ \ldots \\ \left|h_n\right|=r_n.
\end{cases}
\]

If $a_{kj_k}\neq 0$ for any $k=1,\ldots,n$ in a neighborhood of $a_J$, then $\Gamma_h$ is homotopic to the cycle $\Gamma_{q,\tilde a_J}$ defined by \eqref{f6} with the homotopy defined by
\begin{equation}\label{f7}
\begin{cases}
|1-ta_{11}z_1|^{m_{11}}\cdot \ldots \cdot |1-a_{1j_{1}}z_{j1}|^{m_{1j_{1}}} \cdot \ldots \cdot |1-ta_{1n}z_n|^{m_{1n}}=r_1, \\
\ldots \\
|1-ta_{n1}z_1|^{m_{n1}}\cdot \ldots \cdot|1-a_{nj_{n}}z_{jn}|^{m_{nj_{n}}} \cdot \ldots \cdot |1-ta_{nn}z_n|^{m_{nn}}=r_n,
\end{cases}
\end{equation}
where $t\in [0,1]$. For $t=1$ we obtain $\Gamma_h$, and for $t=0$ we obtain $\Gamma_{h,\tilde a_J}$.

Now we consider $\Gamma_h$ in the neighborhood of $a_J$ where $a_{i_1{j_{i_1}}}=\ldots =a_{i_k{j_{i_k}}}=0$. Then in such neighborhood $\Gamma_h$ is homotopic to the cycle $\Gamma_{h,\tilde a_J}$ defined by \eqref{f6}, and the homotopy is defined similarly to \eqref{f7} where for each $a_{i_k{j_{i_k}}}=0$ the corresponding term $1-a_{i_k{j_{i_k}}}z_{j_{i_k}}$ is replaced with $1/z_{i_k}$.
\end{proof}

For
\begin{equation}\label{f8}
F_i(z,t)=q_i(z)+t\cdot Q_i(z), \quad i=1,\ldots,n
\end{equation}
consider the system of equations $F_i(z,t)=0$ which depends on a real parameter $t\geqslant 0$.

Let $r_1>0,\ldots,r_n>0$ be fixed real numbers. Compactness of the cycles $\Gamma_h$ defined by \eqref{f5} yields the fact that for sufficiently small $t>0$, the inequalities
\[
\bigl|q_i(z)\bigr|>\bigl|t\cdot Q_i(z)\bigr|, \ i=1,\ldots,n
\]
hold on $\Gamma_h$.

By $J_\gamma(t)$ we denote the  integral
\begin{equation}\label{f9}
J_\gamma(t)=\frac 1{(2\pi \sqrt{-1})^n}\int\limits_{\Gamma_h}\frac 1{z^{\gamma+I}}\cdot\frac{dF}{F}=\frac 1{(2\pi \sqrt{-1})^n}\int\limits_{\Gamma_h} \frac{1}{z_1^{\gamma_1+1}\cdot\ldots\cdot z_n^{\gamma_n+1}}\cdot\frac{dF_1}{F_1}\wedge\ldots\wedge
\frac{dF_n}{F_n},
\end{equation}
where $\gamma=(\gamma_1,\ldots\gamma_n)$ is a multi-index and $I=(1,\ldots, 1)$. This integral we call the residue integral in accordance with the paper \cite{PTs95}.

In order to formulate the theorem below, we introduce the following notations which we will also use in further statements.

Denote by $\Delta=\Delta(t)$ the Jacobian of the system $F_1(z,t),\ldots,F_n(z,t)$ with respect to $z_1,\ldots, z_n$. Let $(-1)^{s(J)}$ be the sign of the permutation $J$, and $\alpha=(\alpha_1,\ldots, \alpha_n)$ be a multi-index of length $n$. By $q^{\alpha+I}(J)$ we denote $q_1^{\alpha_1+1}[j_1] \cdot\ldots\cdot q_n^{\alpha_n+1}[j_n]$, where $q_s[j_s]$ is a product of all $(1-a_{j1}z_1)^{m_{j1}}\cdot\ldots\cdot (1-a_{jn}z_n)^{m_{jn}}$ except $(1-a_{s{j_s}}z_{s})^{m_{s{j_s}}}$. By $\beta(\alpha, J)$ we denote the vector 
\[
\beta(\alpha, J)=\bigl(m_{1j_{1}}(\alpha_{j_{1}}+1)-1,\ldots,m_{nj_{n}}(\alpha_{j_{n}}+1)-1\bigr),
\]
and 
\[
\beta(\alpha, J)!=\prod_p\bigl(m_{pj_{p}}(\alpha_{j_p}+1)-1\bigr)!
\]
Finally, $a_J^{\beta+I}$ denotes $a_{1{j_1}}^{m_{1j_{1}}(\alpha_{j_{1}}+1)}\cdot\ldots\cdot a_{n j_n}^{m_{nj_{n}}(\alpha_{j_{n}}+1)}$, and
\[
\frac {\partial^{||\beta(\alpha(J)||}}{\partial z^{\beta(\alpha,J)}}=\frac{\partial^{m_{1j_{1}} (\alpha_{j_{1}}+1)-1+\ldots+m_{nj_{n}} (\alpha_{j_{n}}+1)-1}}{\partial z_1^{m_{1j_{1}} (\alpha_{j_{1}}+1)-1} \ldots \partial z_n^{m_{nj_{n}} (\alpha_{j_{n}}+1)-1}}.
\]

\begin{theorem}\label{t1}
Under the assumptions made for the functions $F_i$ defined by \eqref{f8}, the following formulas for $J_\gamma(t)$ as convergent (for sufficiently small $t$) series are valid:
\begin{multline*}
J_\gamma(t)={\sum_J}'  \sum_{\alpha}(-t)^{||\alpha||+\|\beta(\alpha,J)|+n}
 \frac {(-1)^{s(J)}}{\beta(\alpha, J)!\cdot a_J^{\beta+I}} \\ \times \frac{\partial^{||\beta(\alpha(J)||}}{\partial z^{\beta(\alpha,J)}} \left[\frac{\Delta(t)}{ z_1^{\gamma_1+1}\cdot \ldots \cdot z_n^{\gamma_n+1}} \cdot \frac {Q^{\alpha}}{q^{\alpha+I}(J)}\right]_{z=\tilde a_J},
\end{multline*}
where $\sum_J'$ means that the summation is performed over such all multi-indices $J$ such that $a_J$ have no zero components.
\end{theorem}
\begin{proof}
We have
\begin{multline*}
J_\gamma(t)=\frac 1{(2\pi \sqrt{-1})^n}\int\limits_{\Gamma_h}\frac 1{z^{\gamma+I}}\cdot\frac{dF}{F} \\ =\frac 1{(2\pi \sqrt{-1})^n}\int\limits_{\Gamma_h} \frac{1}{z_1^{\gamma_1+1}\cdot\ldots\cdot z_n^{\gamma_n+1}}\cdot\frac{dF_1}{F_1}\wedge\ldots\wedge \frac{dF_n}{F_n} \\
=\frac 1{(2\pi \sqrt{-1})^n}\int\limits_{\Gamma_h}\frac 1{z^{\gamma+I}}\cdot\frac{\Delta dz}{F}, \
\end{multline*}
where $dz=dz_1\wedge\ldots\wedge dz_n$, and $F=F_1\cdot\ldots\cdot F_n$.

Applying a formula for the sum of a geometric sequence, we get
\begin{gather*}
\frac 1{(2\pi \sqrt{-1})^n}\int\limits_{\Gamma_h}\frac 1{{z_1^{\gamma_1+1}\cdot\ldots\cdot z_n^{\gamma_n+1}}}\cdot\frac{\Delta dz}{F_1\cdot\ldots\cdot F_n} \\ 
=\frac 1{(2\pi \sqrt{-1})^n} \sum_{\|\alpha\|\geqslant 0}(-t)^{\|\alpha\|} \int\limits_{\Gamma_{h}} \frac{\Delta(t)}{z_1^{\gamma_1+1}\cdot\ldots\cdot z_n^{\gamma_n+1}}\cdot \frac {Q_1^{\alpha_1}\cdot\ldots\cdot  Q_n^{\alpha_n}}{q_1^{\alpha_1+1}\cdot\ldots\cdot  q_n^{\alpha_n+1}}dz \\ 
 =\frac 1{(2\pi \sqrt{-1})^n} \sum_J\sum_{\|\alpha\|\geqslant 0}(-t)^{\|\alpha\|} \int\limits_{\Gamma_{\Gamma_{h,\tilde a_J}}} \frac{\Delta(t)}{z_1^{\gamma_1+1}\cdot\ldots\cdot z_n^{\gamma_n+1}}\cdot \frac {Q_1^{\alpha_1}\cdot\ldots\cdot  Q_n^{\alpha_n}}{q_1^{\alpha_1+1}\cdot\ldots\cdot  q_n^{\alpha_n+1}}dz \\ 
=\frac 1{(2\pi \sqrt{-1})^n} {\sum_J}'\sum_{\|\alpha\|\geqslant 0}(-t)^{\|\alpha\|} \int\limits_{\Gamma_{\Gamma_{h,\tilde a_J}}} \frac{\Delta(t)}{z_1^{\gamma_1+1}\cdot\ldots\cdot z_n^{\gamma_n+1}}\cdot \frac {Q_1^{\alpha_1}\cdot\ldots\cdot  Q_n^{\alpha_n}}{q_1^{\alpha_1+1}\cdot\ldots\cdot  q_n^{\alpha_n+1}}dz \\ 
=\frac 1{(2\pi \sqrt{-1})^n}{\sum_J}'  (-1)^{s(J)} \sum_{\|\alpha\|\geqslant 0}(-t)^{\|\alpha^s\|}\times\\
\times \int\limits_{\Gamma_{h,\tilde a_J}} \frac{\Delta(t)}{z^{\gamma+I}}\cdot \frac {Q^{\alpha}dz}{q_1^{\alpha_1+1}[j_1] \ldots q_n^{\alpha_n+1}[j_n](1-a_{1j_1}z_{j_1})^{(\alpha_1+1)m_{1j_1}} \ldots (1-a_{nj_n}z_{j_n})^{(\alpha_n+1)m_{nj_n}}},
\end{gather*}
and finally we obtain that $J_\gamma(t)$ is equal to
\[
{\sum_J}'  \sum_{\alpha}(-t)^{||\alpha||+\|\beta\|+n}
 \frac{(-1)^{s(J)}} {\beta(\alpha, J)!\cdot a_J^{\beta+I}} \cdot \frac {\partial^{||\beta(\alpha,J)||}}{\partial z^{\beta(\alpha, J)}} \left[\frac{\Delta(t)}{ z_1^{\gamma_1+1}\cdot \ldots \cdot z_n^{\gamma_n+1}} \cdot \frac {Q^{\alpha}}{q^{\alpha+I}(J)}\right]_{z=\tilde a_J}.
\]

The resulting series converges for sufficiently small $t$.
\end{proof}

\section{Residue integrals and Waring's formulas  for algebraic systems}

In this section we establish a correspondence between the residue integrals and the power sums of the inverses to the roots (Waring's formulas). First we shrink the class of systems for which the sums in Theorem \ref{t1} are finite. Then, applying transformation $z_j=\frac 1{w_j}$, $j=1,\ldots, n$, and Lemma \ref{lts} by A.~Tsikh we rewrite the residue integrals $J_\gamma(t)$ in new variables $w$ (Lemma \ref{l2}). Further, Lemma \ref{l3} shows that $J_\gamma(t)$ can be expressed by a finite number of Taylor coefficients of the considered functions. Theorem \ref{t2} shows (by means of Lemma \ref{l4}) that the residue integral $J_\gamma(t)$ equals (up to a sign) to the power sums of the inverses to the roots. The main result of the paper, Theorem \ref{t3}, shows that the statement of Theorem \ref{t2} is true not only for sufficiently small $t>0$ but also for $t=1$.  (Note that Theorems \ref{t3} and \ref{t4} allow one find power sums of the inverses to the roots of the systems without finding the roots.) As a conclusion of the section we present elimination method for the considered systems.

Suppose $Q_i(z)$ are the polynomials:
\begin{equation}\label{f10}
Q_i(z)=z_1\cdot\ldots\cdot z_n\sum_{|\alpha\|\geq 0} C^i_{\alpha}z^{\alpha} \quad i=1,\ldots,n,
\end{equation}
where $\alpha$ is a multi-index,
$z^{\alpha}=z_1^{\alpha_1}\cdot\ldots\cdot z_n^{\alpha_n}$, and $\deg_{z_j}Q_i\leqslant m_{ij}$, $i,j=1,\ldots,n$ for all non-zero $a_{ij}$. If $a_{ij}=0$ then no restriction on $\deg_{z_j}Q_i$ is needed.

Assuming that all $w_j\neq 0$, we substitute $z_j=\dfrac{1}{w_j}$, $j=1,\ldots,n$ in the functions 
\[
F_i(z,t)=\bigl(q_i(z)+t\cdot Q_i(z)\bigr),\quad i=1,\ldots,n.
\]
Consequently, for $i=1,\ldots,n$, we get
\begin{gather*}
F_i\left(\frac 1{w_1},\ldots, \frac 1{w_n}, t\right)=q_i\left(\frac 1{w_1},\ldots, \frac 1{w_n}\right)+
t \cdot Q_i\left(\frac 1{w_1},\ldots, \frac 1{w_n}\right) \\ 
=\left(1-a_{i1}\frac 1{w_1}\right)^{m_{i1}}\cdot \ldots \cdot \left(1-a_{in}\frac 1{w_n}\right)^{m_{in}}
+t \cdot Q_i\left(\frac 1{w_1},\ldots, \frac 1{w_n}\right) \\ 
=\left(\frac 1{w_1}\right)^{m_{i1}}\cdot \ldots \cdot\left(\frac 1{w_n}\right)^{m_{in}}\cdot(w_1-a_{i1})^{m_{i1}}\cdot \ldots \cdot(w_n-a_{in})^{m_{in}}
+t \cdot Q_i\left(\frac 1{w_1},\ldots, \frac 1{w_n}\right).
\end{gather*}
And finally we arrive at
\begin{equation}\label{f11}
F_i\left(\frac 1{w_1},\ldots, \frac 1{w_n},t\right)=\left(\frac 1{w_1}\right)^{m_{i1}}\cdot \ldots \cdot\left(\frac 1{w_n}\right)^{m_{in}}\cdot \left(\widetilde q_i(w)+t \cdot \widetilde Q_i(w)\right),
\end{equation}
where $\widetilde q_i$ are the functions 
\[
\widetilde q_i=(w_1-a_{i1})^{m_{i1}}\cdot \ldots \cdot(w_n-a_{in})^{m_{in}},
\]
and $\widetilde Q_i$ are the polynomials 
\[
\widetilde Q_i=w_1^{m_{i1}}\cdot \ldots \cdot w_n^{m_{in}}\cdot Q_i\left(\dfrac 1{w_1},\ldots, \dfrac 1{w_n}\right).
\]
From \eqref{f10} we obtain 
\[
\deg_{w_j} \widetilde Q_i < m_{ij}, \ i,j=1,\ldots, n.
\]

Note that in the above calculations it is not important whether $a_{ij}$ vanish or not. Indeed, suppose that in $F_i(z,t)=q_i(z)+t \cdot Q_i(z)$, $i=1,\ldots,n$, some $a_{ij}=0$. If, for instance, $a_{11}=0$, then after the substitution $z_j=\dfrac 1{w_j}$, $j=1,\ldots,n$, the function $F_1$ takes the form
\[
F_1\left(\frac 1{w_1},\ldots, \frac 1{w_n}, t\right)=q_1\left(\frac 1{w_1},\ldots, \frac 1{w_n}\right)+t \cdot Q_1\left(\frac 1{w_1},\ldots, \frac 1{w_n}\right),
\]
where
\begin{multline*}
q_1\left(\frac 1{w_1},\ldots, \frac 1{w_n}\right)=\left(1-a_{12}\frac 1{w_2}\right)^{m_{12}}\cdot \ldots \cdot \left(1-a_{1n}\frac 1{w_n}\right)^{m_{1n}} \\ 
=\left(\frac 1{w_1}\right)^{\deg_{w_1} Q_1}\cdot\left(\frac 1{w_2}\right)^{m_{12}}\cdot \ldots \cdot\left(\frac 1{w_n}\right)^{m_{1n}} \\
\times w_1^{\deg_{w_1} Q_1}\cdot(w_2-a_{12})^{m_{12}}\cdot \ldots \cdot(w_n-a_{1n})^{m_{1n}}.
\end{multline*}
Consequently
\[
F_1\left(\frac 1{w_1},\ldots, \frac 1{w_n}, t\right)=\left(\frac 1{w_1}\right)^{\deg_{w_1} Q_1}\cdot \ldots \cdot\left(\frac 1{w_n}\right)^{m_{1n}}\cdot \left(\widetilde q_1(w)+t \cdot \widetilde Q_1(w)\right),
\]
where 
\[
\widetilde q_1=(w_1)^{\deg_{w_1} Q_1}\cdot(w_2-a_{12})^{m_{12}}\cdot \ldots \cdot(w_n-a_{1n})^{m_{1n}},
\]
and 
\[
\widetilde Q_1=w_1^{\deg_{w_1} Q_1}\cdot w_2^{m_{12}}\cdot \ldots \cdot w_n^{m_{1n}}\cdot Q_1\left(\dfrac 1{w_1},\ldots \dfrac 1{w_n}\right).
\]
I.e., we can take $m_{11}=\deg_{w_1} Q_1$. From \eqref{f10} we derive that
\[
\deg_{w_j} \widetilde Q_1 < m_{1j}, \ j=1,\ldots, n.
\]

Denote
\begin{equation}\label{f12}
\widetilde{F_i}=\widetilde{F_i}(w,t)=\widetilde q_i(w)+t \cdot \widetilde Q_i(w), \quad i=1,\ldots,n.
\end{equation}

When $0\leqslant t \leqslant 1$, the system \eqref{f12} has finite number of zeros in $\mathbb C^n$ which depend on $t$. Moreover, \eqref{f12} does not have infinite roots in $\overline{\mathbb C}^n$ (see \cite{Ts80} and \cite[Theorems~8.5,~8.6]{BKL98}). As it was shown in \cite{Ts80} (see also \cite[Theorem~8.5]{BKL98}) the number of zeros (counting multiplicities) is equal to the permanent of the matrix $(m_{ij})_{1\leqslant i,j\leqslant n}$.

Consider the cycle 
\[
\widetilde \Gamma_h=\left\{w\in\mathbb C^n:\left|h_i\left(\dfrac 1{w_1},\ldots,\dfrac 1{w_n}\right)\right|=\varepsilon_i, \quad i=1,\ldots,n\right\}
\]
for small enough $t$.

Compactness of the cycle $\widetilde \Gamma_h$ implies that
\[
\left|q_i\left(\dfrac 1{w_1},\ldots,\dfrac 1{w_n}\right)\right|>\left|t \cdot Q_i\left(\dfrac 1{w_1},\ldots,\dfrac 1{w_n}\right)\right|, \quad i=1,\ldots,n.
\]

Therefore, $\widetilde \Gamma_{h}$ is homologous to the sum of the cycles $\widetilde \Gamma_{h, \tilde a_J}$ 
\[
\begin{cases}
\left|1-a_{1i_{1}}\frac 1{w_1}\right|=\varepsilon_1, \\ \ldots \\ \left|1-a_{ni_{n}}\frac 1{w_n}\right|=\varepsilon_n,
\end{cases}
\]
obtained from the cycles $\Gamma_{h, \tilde a_J}$ by the substitution $z_j=\dfrac 1{w_j}$.

The equation 
\[
\left|1-a_{ji_{j}}\dfrac 1{w_j}\right|=\varepsilon
\]
defines a circle. Indeed, let us first rewrite it in the form
\[
|w_j-a_{ji_{j}}|=\varepsilon|w_j| \quad \text{or}\quad |w_j-a_{ji_{j}}|^2=\varepsilon^2|w_j|^2.
\]
Thus
\[
(1-\varepsilon^2)\left|w_j-\frac{a_{ji_{j}}}{1-\varepsilon^2}\right|^2=\frac{\varepsilon^2\cdot|a_{ji_{j}}|^2}{(1-\varepsilon^2)},
\]
or
\[
\left|w_j-\frac{a_{ji_{j}}}{1-\varepsilon^2}\right|^2=\frac{\varepsilon^2\cdot|a_{ji_{j}}|^2}{(1-\varepsilon^2)^2},\quad j=1,\ldots, n.
\]
For sufficiently small $\varepsilon$ the point $a_{ji_{j}}$ lies inside this circle, and, therefore, $\widetilde \Gamma_{h, \tilde a_J}$ is homologous to the cycle $\widetilde \Gamma_{h,  a_J}$:
\[
\begin{cases}
\left|w_1-a_{1j_{1}}\right|=\varepsilon_1,\\
\ldots \\
\left|w_n-a_{nj_{n}}\right|=\varepsilon_n.
\end{cases}
\]
Here $a_{kj_{k}}$ can vanish for some $k$.

\begin{lemma}\label{l2} The following formula holds for the residue integral \eqref{f9}:
\[
J_\gamma(t)=\frac{(-1)^n}{(2\pi \sqrt{-1})^n}\int\limits_{\widetilde \Gamma_h} w_1^{\gamma_1+1}\cdot\ldots\cdot
w_n^{\gamma_n+1}\cdot\frac{d\widetilde{F_1}}{\widetilde{F_1}}\wedge\ldots\wedge \frac{d{\widetilde F_n}}{\widetilde{F_n}}.
\]
\end{lemma}
\begin{proof}
The equality \eqref{f11} yields
\[
F_j\left(\frac 1{w_1},\ldots, \frac 1{w_n},t\right)=\left(\frac 1{w_1}\right)^{m_{i1}}\cdot\ldots\cdot \left(\frac 1{w_n}\right)^{m_{in}}\cdot \widetilde{F_j}(w,t), \quad j=1,\ldots, n.
\]

Then
\[
\frac{dF_j\left(\frac 1{w_1},\frac 1{w_2},\ldots,\frac 1{w_n}, t\right)}{
F_j\left(\frac 1{w_1},\frac 1{w_2},\ldots,\frac 1{w_n},t\right)}=
\frac{d{\widetilde F_j}(w,t)}{{\widetilde F_j}(w,t)}
-\sum_{k=1}^n m_{jk}\cdot\frac{dw_k}{w_k}.
\]

Using \eqref{f11} and taking into account the change of orientation of the space after replacing $z_j=1/w_j$, $j=1,\ldots,n$, one can rewrite the integral $J_\gamma(t)$ as
\begin{gather*}
J_\gamma(t)=\frac{(-1)^n}{(2\pi \sqrt{-1})^n}\int\limits_{\widetilde \Gamma_h}w^{\gamma+I}\cdot\frac{dF\left(\frac 1{w_1},\ldots,\frac 1{w_n}, t\right)}{F\left(\frac 1{w_1},\ldots,\frac 1{w_n}, t\right)} \\ 
=\frac{(-1)^n}{(2\pi \sqrt{-1})^n}\int\limits_{\widetilde\Gamma_h} w_1^{\gamma_1+1}\ldots w_n^{\gamma_n+1}\cdot\frac{dF_1}{F_1}\wedge\ldots\wedge\frac{dF_n}{F_n} \\ 
=\frac{(-1)^n}{(2\pi \sqrt{-1})^n}\int\limits_{\widetilde\Gamma_h}w^{\gamma+I}\left(
\frac{d{\widetilde F_1}(w)}{{\widetilde F_1}(w) }-\sum_{k=1}^n m_{1k}\cdot\frac{dw_k}{w_k}\right)\wedge\ldots 
\wedge\left(\frac{d{\widetilde F_n}(w)}{{\widetilde F_n}(w) }-\sum_{k=1}^n m_{nk}\cdot\frac{dw_k}{w_k}\right).
\end{gather*}
All the integrals
\begin{equation}\label{f13}
\int\limits_{\widetilde\Gamma_h}w^{\gamma+I} \frac{d{\widetilde F_{i_1}}(w)}{{\widetilde F_{i_1}}(w) }
\wedge\ldots\wedge\frac{d{\widetilde F_{i_l}}(w)}{{\widetilde F_{i_l}}(w) }\wedge
\frac{dw_{j_1}}{w_{j_1}}\wedge\ldots\wedge \frac{dw_{j_{n-l}}}{w_{j_{n-l}}}
\end{equation}
vanish when $0\leqslant l<n$ and $\varepsilon_j$ are large enough.

Indeed, when $\varepsilon_j$, $j=1,\ldots,n$ are sufficiently large, the inequalities
\[
|\widetilde{q}_j|>|t\cdot \widetilde Q_j(w)|
\]
hold on $\widetilde{\Gamma}_h$. Therefore
\begin{equation}\label{f14}
\frac 1{\widetilde F_j(w)}=\sum_{p=0}^\infty\frac{(-1)^p t^p\widetilde Q_j^p(w)}{\widetilde q_j^{(p+1)}}.
\end{equation}
Consequently, the integrals \eqref{f13} are absolutely convergent series of integrals
\[
\int\limits_{\widetilde{\Gamma}_h}w^{\gamma+I}\frac{w^\alpha\, dw_1\wedge\ldots\wedge dw_n}{\widetilde q_1^{(p_1+1)}\ldots \widetilde q_{i_l}^{(p_l+1)}\cdot w_{j_{i_1}}\ldots w_{j_{n-l}}}.
\]
Stokes' theorem and the fact that all the integrands are holomorphic imply vanishing of all these integrals. 

Finally, we arrive at
\[
J_\gamma(t)=\frac{(-1)^n}{(2\pi \sqrt{-1})^n}\int\limits_{\widetilde\Gamma_h} w_1^{\gamma_1+1}\ldots
w_n^{\gamma_n+1}\cdot\frac{d\widetilde F_1}{\widetilde F_1}\wedge\ldots\wedge\frac{d\widetilde F_n}{\widetilde F_n}
\]
\end{proof}

Now we state result from \cite{Ts80} that we will need for further discussion.

Consider a system of algebraic equations in $\mathbb{C}^n$:
\begin{equation}\label{f15}
f_i(z)=0, \quad i=1,\ldots, n.
\end{equation}
Suppose \eqref{f15} has finite number of roots in $\mathbb{C}^n$ and does not have infinite roots in $\overline{\mathbb C}^n$.

Denote $m_{ij}=\deg_{z_j}f_i$. When $r_1,\ldots, r_n$ are sufficiently small, the cycle
\[
\Gamma=\{z\in\mathbb C^n:\ |f_1(z)|=r_1,\ldots, |f_n(z)|=r_n\}
\]
is homologous to the sum of cycles lying in the neighborhood of the roots of \eqref{f15}.

\begin{lemma}[A.~Tsikh, \cite{Ts80}]\label{lts}
Under the above assumptions
\[
\int_{\Gamma}\frac{P(z) dz}{f_1(z)\ldots f_n(z)}=0
\]
for any polynomial $P(z)$ such that $p_j=\deg_{z_j}P\leqslant m_{1j}+\ldots +m_{nj}$ for all $j=1,\ldots,n$.
\end{lemma}

The Lemma was proved using the residue theorem (theorem on a total sum of residues) on a compact complex manifold.

\begin{lemma}\label{l3}
Let $\widetilde \Delta=\widetilde \Delta(w,t)$ be the Jacobian of \eqref{f12} with respect to $w_1,\ldots, w_n$. Then
\[
J_\gamma(t)=\sum_{K\in\Re}(-t)^{||K||+n} \sum_J  \frac {(-1)^{s(J)}}{\beta(K, J)!} \cdot \frac {\partial^{||\beta(K,J)||}}{\partial w^\beta(K,J)}
\left[ \widetilde \Delta \cdot w_1^{\gamma_1+1}\cdot\ldots\cdot  w_n^{\gamma_n+1} \cdot \frac {\widetilde Q^K}{\widetilde q^{K+I}(J)}\right]_{w=a_J},\
\]
where $K=(k_1,\ldots, k_n)$ is multi-index,  $\widetilde Q^K=\widetilde Q_1^{k_1}\cdot\ldots\cdot \widetilde Q_n^{k_n}$, 
\[
\beta(K, J)=\bigl(m_{1j_{1}}(k_{j_{1}}+1)-1,\ldots,m_{nj_{n}}(k_{j_{n}}+1)-1\bigr),
\]
$\beta(K, J)!=\prod\limits_p\bigl(m_{pj_{p}}(k_{j_p}+1)-1\bigr)!$, and
\[
\Re=\{ K=(k_1,\ldots,k_n)\colon \text{there exists}\ \gamma_i\ \text{such that}\ \|K\|<\gamma_i+2, \ i=1,\ldots,n\}.
\]
\end{lemma}
The rest of the notations in the statement are as in Theorem \ref{t1}.
\begin{proof}
Representation \eqref{f14} and Lemma \ref{l2} yield 
\begin{multline*}\label{f16}
J_\gamma(t)=\frac{(- 1)^n}{(2\pi \sqrt{-1})^n}\int\limits_{\widetilde\Gamma_h} w_1^{\gamma_1+1}\cdots
w_n^{\gamma_n+1}\cdot\frac{d{\widetilde F_1}}{{\widetilde F_1}}\wedge\ldots\wedge
\frac{d{\widetilde F_n}}{{\widetilde F_n}} \\ 
=\frac {(-1)^n}{(2\pi \sqrt{-1})^n}\int\limits_{\widetilde\Gamma_h} w_1^{\gamma_1+1}\cdots
w_n^{\gamma_n+1}\sum_{\|K\|\geqslant 0} (-t)^{||K||} \widetilde \Delta \cdot \frac {\widetilde Q_1^{k_1}\cdots \widetilde Q_n^{k_n}}{\widetilde q_1^{k_1+1}\cdots \widetilde q_n^{k_n+1}}dw \\
=\frac 1{(2\pi \sqrt{-1})^n}\sum_{\|K\|\geqslant 0}(-t)^{\|K\|+n} \sum_J (-1)^{s(J)} \int\limits_{\widetilde\Gamma_{h,a_J}} \widetilde \Delta \cdot w_1^{\gamma_1+1}\cdots w_n^{\gamma_n+1}\cdot \frac {\widetilde Q_1^{k_1}\cdots \widetilde Q_n^{k_n}}
{\widetilde q_1^{k_1+1}\cdots \widetilde q_n^{k_n+1}}dw,
\end{multline*}
so that
\begin{equation}\label{f16}
J_\gamma(t)=\sum_{\|K\|\geqslant 0}(-t)^{\|K\|+n} \sum_J  \frac{(-1)^{s(J)}}{\beta(K, J)!} \cdot \frac {\partial^{||\beta(K, J)||}}{\partial w^\beta(K, J)}
\left[ \widetilde \Delta \cdot w_1^{\gamma_1+1}\cdots w_n^{\gamma_n+1}
\cdot \frac {\widetilde Q^K}{\widetilde q^{K+I}(J)}\right]_{w=a_J}. 
\end{equation}

We now show that the summation in the above formulas is over a finite set of multi-indices. To show this, we estimate the degrees in $w_i$ of the numerator and compare with the corresponding degrees of the denominator of \eqref{f16}.

The degree of the numerator in $w_i$ is less than or equal to
\[
p_i=m_{1i}+\ldots+m_{ni}-1 + \gamma_i+1 + (m_{1i}-1)k_1+\ldots+(m_{ni}-1)k_n.
\]

The corresponding degree of the denominator is
\[
s_i=m_{1i}(k_1+1)+\ldots+m_{ni}(k_n+1).
\]

Lemma \ref{lts} implies vanishing of all the integrals for which the inequality $p_i\leqslant s_i-2$ holds for all $i=1,\ldots,n$, so that
\begin{multline*}
m_{1i}+\ldots+m_{ni}-1 + \gamma_i+1 + (m_{1i}-1)k_1+\ldots+(m_{ni}-1)k_n \\ \leqslant m_{1i}(k_1+1)+\ldots+m_{ni}(k_n+1) - 2.
\end{multline*}
After combining like terms we arrive at
\[
\gamma_i+1-k_1-\ldots-k_n-1 \leqslant -2
\]
or
\[
\gamma_i+2 \leqslant \|K\|.
\]

Thus, the only non-zero integrals in \eqref{f16} are the ones for which $K$ runs over such set that $\gamma_i+2>\|K\|$ for at least one $\gamma_i$.
\end{proof}

\begin{lemma}\label{l4}
Let $w_1, \ldots, w_s$ where $w_j=( w_{j1}, \ldots,  w_{jn})$, $j=1,\ldots,n$ be all the zeros (depending on $t$) of \eqref{f12} counting multiplicities. Then
\[
J_\gamma=(-1)^n\sum_{j=1}^s  w_{j1}^{\gamma_1+1}\cdot w_{j2}^{\gamma_2+1}\cdots w_{jn}^{\gamma_n+1}.
\]
\end{lemma}
The number $s$ of zeros is equal to the permanent of the matrix $(m_{ij})_{1\leqslant i,j\leqslant n}$ (see \cite{Ts80} or \cite[Theorem~8.5]{BKL98}).
\begin{proof}
The statement follows from the multidimensional logarithmic residue formula and the theorem on shifted skeleton (see \cite[Chapter~3]{AY83}).
\end{proof}

Denote by $z^{(j)}(t)=(z_{j1}(t),\ldots,z_{jn}(t))$, $j=1,\ldots,s,$ the zeros of \eqref{f2} with the functions $tQ_i$, where $Q_i$ are defined by \eqref{f10} and do not lie in the coordinate subspaces. The latter fact implies that the number of the zeros is finite. Since $w_j$ do not lie in coordinate subspaces, then $z_{jm}=1/w_{jm}$, $m=1,\ldots,n$ and therefore we have finite number $p$ of zeros. Consequently $p\leqslant s$.

\begin{theorem}\label{t2}
The following equality holds:
\begin{gather*}
\sum_{j=1}^p \frac{1}{z_{j1}(t)^{\gamma_1+1}\cdots z_{jn}(t)^{\gamma_n+1}} \\ 
=\sum_{K \in \Re}(-t)^{||K||+n} \sum_J  \frac {(-1)^{s(J)}}{\beta(K, J)!} \cdot \frac {\partial^{||\beta(K, J)||}}{\partial w^\beta(K, J)} \left[
\widetilde \Delta(t)\cdot
w_1^{\gamma_1+1}\cdots w_n^{\gamma_n+1} \cdot \frac {\widetilde Q^K}{\widetilde q^{K+I}(J)}\right]_{w=a_J}.
\end{gather*}
\end{theorem}
\begin{proof}
The statement follows from Lemmas \ref{l3} and \ref{l4}.
\end{proof}

Thus, the power sum of  {zeros of \eqref{f12}} is a polynomial on $t$, and, therefore, the equality in Theorem~\ref{t2} also holds for $t=1$.

Denote 
\[
\sigma_{\gamma+I}=\sum\limits_{j=1}^p \frac{1}{z_{j1}^{\gamma_1+1}\cdots z_{jn}^{\gamma_n+1}},
\]
where $z^{(j)}=(z_{j1},\ldots,z_{jn})=(z_{j1}(1),\ldots,z_{jn}(1))$, $j=1,\ldots,p$.

\begin{theorem}[Waring's formulas]\label{t3}
For the system \eqref{f2} with functions $f_j$ defined by \eqref{f2} and $Q_i$ defined by \eqref{f10} the following formulas are valid:
\begin{gather*}
\sigma_{\gamma+I}=\sum_{j=1}^p \frac{1}{z_{j1}^{\gamma_1+1}\cdots z_{jn}^{\gamma_n+1}} \\ 
=\frac 1{(2\pi \sqrt{-1})^n}\sum_{\|K\|\geqslant 0}(-1)^{\|K\|+n} \sum_J (-1)^{s(J)} \int\limits_{\widetilde\Gamma_{h,a_J}} \widetilde \Delta \cdot w_1^{\gamma_1+1}\ldots w_n^{\gamma_n+1}\cdot \frac {\widetilde Q_1^{k_1}\cdots \widetilde Q_n^{k_n}}
{\widetilde q_1^{k_1+1}\cdots \widetilde q_n^{k_n+1}}dw \\ 
=\sum_{K \in \Re}(-1)^{||K||+n} \sum_J  \frac{(-1)^{s(J)}} {\beta(K, J)!} \cdot \frac {\partial^{||\beta(K, J)||}}{\partial w^\beta(K, J)} \left[
\widetilde \Delta\cdot
w_1^{\gamma_1+1}\cdots w_n^{\gamma_n+1} \cdot \frac {\widetilde Q^K}{\widetilde q^{K+I}(J)}\right]_{w=a_J}.
\end{gather*}
\end{theorem}
\begin{proof}
The statement of the Theorem is a corollary of Theorem \ref{t2}.
\end{proof}

Note that in \cite{KS15} the authors considered algebraic systems and obtained expansions of their solutions into geometric series. Moreover, the authors obtained analogs of Waring's formulas for the systems
\[
y_j^{m_j} +\sum_{\lambda\in\Lambda^{(j)}\cup \{0\}} x_{\lambda}^{(j)} y^\lambda =0, \quad \lambda_1+\ldots +\lambda_n<m_j, \quad j=1,\ldots, n,
\]
where leading homogeneous parts are monomials (here $\Lambda^{(j)}$ is a finite  set of multi-indices).

\section{Residue integrals and Waring's formulas for transcendental systems}

We now consider a more general situation. Let $f_j$ be entire functions in $\mathbb{C}^n$ of finite order not greater than $\rho$ and 
\begin{equation}\label{f17}
f_j(z)=\prod_{s=1}^\infty f_{j,s}(z),\quad j=1,\ldots,n.
\end{equation}
Here $f_{j,s}(z)$ are entire functions in $\mathbb{C}^n$ of finite order not greater than $\rho$ admitting expansion into uniformly convergent in $\mathbb{C}^n$ infinite products with factors of the form
\[
f_{j,s}(z)=\bigl(q_{j,s}+Q_{j,s}(z)\bigr),
\]
where $q_{j,s}(z)$ and $Q_{j,s}(z)$ are the polynomials defined by \eqref{f3} and \eqref{f10} respectively.
An entire function of several complex variables is not always decomposable into an infinite product of the functions associated with its zeros (see, e.g., \cite{LG86}). The sufficient conditions for the existence of such an expansion (in the form of convergence of the distances between the origin and zero sets of the functions $q_{j,s}+Q_{j,s}(z)$) one can find in \cite{Mysh14}.

Denote by $z^{(j)}=(z_{j1},\ldots,z_{jn})$, $j=1,\ldots,\infty$, the zeros of \eqref{f17} not lying on coordinate subspaces,  counting multiplicities.

We now give a multidimensional Waring's formula for transcendental systems.

\begin{theorem}\label{t4} 
Consider the system $f_j(z)=0$, $j=1,\ldots,n$, with functions defined by \eqref{f17}. Then, the following formulas are valid:
\begin{gather*}
\sigma_{\gamma+I}=
\sum_{j=1}^\infty \frac{1}{z_{j1}^{\gamma_1+1}\cdots z_{jn}^{\gamma_n+1}} \\
=\sum_{K \in \Re}(-1)^{||K||+n}\sum_S \sum_J  \frac {(-1)^{s(J)}}{\beta(K, J)!} \cdot \frac {\partial^{||\beta(K, J)||}}{\partial w^{\beta(K, J)}} \left[
\widetilde \Delta \cdot
w_1^{\gamma_1+1}\cdots  w_n^{\gamma_n+1} \cdot \frac {\widetilde Q^K(s)}{\widetilde q^{K+I}(J,s)}\right]_{w=a_J},
\end{gather*}
where $\widetilde Q^K(s)=\widetilde Q_{1,s}^{k_1}\cdots \widetilde Q_{n,s}^{k_n}$.
\end{theorem}
\begin{proof}
We have
\[
\frac{d\, f_j(z)}{f_j(z)}=\frac{d\prod\limits_{s=1}^\infty f_{js}(z)}{\prod\limits_{s=1}^\infty f_{js}(z)}=\sum_{s=1}^\infty \frac{d\, f_{js}(z)}{f_{js}(z)}.
\]
It is easy to check that the above series converges uniformly on $\widetilde\Gamma_{h,a_J}$. Indeed, if a sequence of continuous on a compact set $M$ functions $f_m$ uniformly convergent to a function $f$ on $M$ and $f\neq 0$ on $M$ is given, then, starting from some number $m_0$ for $m\geqslant m_0$, the functions $f_m\neq 0$ on $M$ and the sequence $1/f_m$ will converge uniformly on $M$ to the function $1/f$. Similarly, one can show that uniform convergence is preserved under element-wise multiplication.

Uniform convergence of $\prod\limits_{s=1}^\infty f_{js}(z)$ on $\widetilde\Gamma_{h,a_J}$ yields uniform convergence of the series
\[
\sum_{s=1}^\infty \frac{d\, f_{j,s}(z)}{f_{j,s}(z)}=\frac{d\prod\limits_{s=1}^\infty f_{j,s}(z)}{\prod\limits_{s=1}^\infty f_{j,s}(z)}=\lim_{m\to \infty}\frac{d\prod\limits_{s=1}^m f_{j,s}}{\prod\limits_{s=1}^m f_{j,s}}
\]
on $\Gamma_q$. Thus, $J_\gamma(1)$ is defined and is equal to the convergent series of the integrals
\[
\frac 1{(2\pi \sqrt{-1})^n}\int\limits_{\Gamma_q}\frac 1{z^{\gamma+I}} \cdot \frac{d\, f_{1,s_1}(z)}{f_{1s_1}(z)}\wedge\ldots\wedge \frac{d\, f_{ns_n}(z)}{f_{n,s_n}(z)},
\]
where the summation is taken over the cubes with integer sides centered at the origin. And, for each such integral the required formula was proved in Theorem \ref{t3}.

If $\prod\limits_{s=1}^\infty f_{j,s}(z)$ converge absolutely, then their values does not depend on a permutation of their factors. In other words, changing the numbering of the roots does not affect the values of the infinite products of $f_{j,s}(z)$. Consequently, $J_\gamma$ also does not depend on the permutation of its terms. This yields absolute convergence of $J_\gamma$ and $\sigma_{\gamma+I}$.
\end{proof}

\noindent{\bf Remark 3.1} We are now ready to describe the scheme of elimination of unknowns. Consider a system of equations of the form as in Theorems \ref{t3} and \ref{t4}. Let $s_i$ be the power sums of its roots
\begin{equation}\label{f18}
s_i=\sigma_{i,\ldots, i}=\sum_{j=1}^\infty \frac{1}{z_{j1}^{i}\cdot\ldots\cdot z_{jn}^{i}}, \quad  i\geqslant 1. 
\end{equation}

Now, we need to find an entire function $f(w)$ of a single variable $w\in\mathbb{C}$, such that the power sums of its roots coincide with $s_i$ (by Weierstrass theorem). Let the Taylor expansion of this function be
\[
f(w)=1+b_1 w+\ldots +b_k w^k +\ldots
\]
Since the series in the right hand side of \eqref{f18} converge absolutely, then $f$ can be decomposed into an infinite product with respect to its zeros 
\[
c_i=\frac{1}{z_{j1}^{i}\cdot\ldots\cdot z_{jn}^{i}}, \ i\geqslant 1
\]
(Hadamard's formula), which yields that $f(w)$ is an entire function of at most first order of growth. The analogs of recurrent Newton formulas connecting the coefficients $b_k$ and the sums $s_i$ for such functions were given in \cite[Chapter~1]{BKL98}. More precisely, Theorem~2.3 in \cite{BKL98} states that
\[
\sum_{j=0}^{k-1}b_j s_{k-j}+kb_k=0, \quad b_0=1, \quad k\geqslant 1.
\]
These formulas allow one to find the coefficients of the function $f(w)$, whose roots are $c_i$. So, the function $f(w)$ is an analog of the resultant for a system of algebraic equations.

\section{Examples}

Since the power sums in Theorem \ref{t4} are multidimensional series, then, clearly, this theorem provides one with a method for computing multidimensional series of such kind. For this purpose, one has to find a system such that the power sums of its roots coincide with elements of the given series. Once such a system has been found, the sum of the series can then be found by using Theorem \ref{t4}.

In this section, we consider examples which admit use of the described method. One can find power sums of roots by applying Theorem \ref{t3}. Then, in the second example, we consider two functions each of which admits expansion into an infinite product of the functions considered in the first example. Using Theorem \ref{t4} we then construct the series consisting of the power sums of the roots of the system and find the sum of this series.

\begin{example}
Consider the system in two complex variables
\begin{equation}\label{f19}
\begin{cases} f_1(z_1,z_2)=(1-a_2z_2)^2+a_3z_1z^2_2=0, \\ f_2(z_1,z_2)=(1-b_1z_1)^2(1-b_2z_2)+b_3z^2_1z_2=0 \end{cases}
\end{equation}
with real coefficients $a_i$ and $b_i$. For this system, $Q_1$ and $Q_2$ are of the form \eqref{f10}. It is not hard to verify that the system \eqref{f19} has 5 roots $(z_{j1},z_{j2})$, $j=1,2,3,4,5$. If $a_2\neq b_2$, then all the roots do not lie in the coordinate hyperplanes.

After the substitution $z_1=1/w_1$, $z_2=1/w_2$, \eqref{f19} takes the form
\begin{equation}\label{f20}
\begin{cases} \widetilde f_1=w_1(w_2-a_2)^2+a_3=0, \\ \widetilde f_2=(w_1-b_1)^2(w_2-b_2)+b_3=0. \end{cases}
\end{equation}

The Jacobian $\widetilde{\Delta}$ of \eqref{f20} is equal to
\begin{multline*}
\widetilde \Delta=\vmatrix (w_2-a_2)^2 & 2 w_1(w_2-a_2) \\ 2(w_1-b_1)(w_2-b_2) & (w_1-b_1)^2 \endvmatrix \\ 
= (w_1-b_1)^2 (w_2-a_2)^2 - 4 w_1(w_1-b_1)(w_2-a_2)(w_2-b_2).
\end{multline*}

Now, using Theorem \ref{t3}, we compute the power sums
\begin{multline*}
\sigma_{\gamma}=\sum_{j=1}^5 \frac {1}{z^{\gamma_1+1}_{j1}}\cdot \frac {1}{z^{\gamma_2+1}_{j2}} \\
=\sum_J (-1)^{s(j)}\sum_{K \in \Re} \frac {(-1)^{\|K\|}}{(2\pi i)^2} \int\limits_{\widetilde \Gamma_{h,{a_J}}}\frac {w_1^{\gamma_1+1} \cdot w_2^{\gamma_2+1}\cdot a_3^{k_1} \cdot b_3^{k_2} \cdot \widetilde \Delta \cdot dw_1 \land dw_2}{w_1^{k_1+1}(w_2-a_2)^{2(k_1+1)}\cdot (w_1-b_1)^{2(k_2+1)}(w_2-b_2)^{k_2+1}}.
\end{multline*}
Here $\Re=\{ K=(k_1,k_2)\colon \text{there exists}\ i\ \text{such that}\ \gamma_i+2 > k_1+k_2\ \text{for}\ i=1,2\}$, and $\widetilde \Gamma_{h,a_J}$ are the cycles either $\{|w_1|=r_{11}, |w_2-b_2|=r_{22}\}$ oriented positively or $\{|w_2-a_2|=r_{12}, |w_1-b_1|=r_{21}\}$ oriented negatively.

In particular, by computing $J_{(0,0)}$ and performing necessary algebra, we obtain that
\begin{equation}\label{f21}
\sigma_{(1,1)}=4a_2b_1-\frac{a_3b_2}{(b_2-a_2)^2}
\end{equation}
without finding the roots.
\end{example}

\begin{example}
Recall the known expansions of $\sin{z}$ into an infinite product and a power series:
\[
\frac{\sin\sqrt z}{\sqrt z}=\prod_{k=1}^\infty\left(1-\frac z{k^2\pi^2}\right)=\sum_{k=0}^\infty \frac{(-1)^kz^k}{(2k+1)!}.
\]
Both are absolutely and uniformly convergent on any compact subset of the complex plane.

Consider the system
\begin{equation}\label{f22}
\begin{cases} f_1=\dfrac{\sin\sqrt {2a_2z_2-a^2_2z^2_2-a_3z_1z^2_2}}{\sqrt{2a_2z_2-a^2_2z^2_2-a_3z_1z^2_2}}=
\prod\limits_{k=1}^\infty\left(\left(1-\dfrac{a_2z_2}{k^2\pi^2}\right)^2+\dfrac{a_3z_1z^2_2}{k^2\pi^2}\right)=0,\\
f_2=\dfrac{\sin\sqrt{b_2z_2+2b_1z_1-2b_1b_2z_1z_2-b^2_1z^2_1+b^2_1b_2z^2_1z_2-b_3z^2_1z_2}}
{\sqrt{b_2z_2+2b_1z_1-2b_1b_2z_1z_2-b^2_1z^2_1+b^2_1b_2z^2_1z_2-b_3z^2_1z_2}}\\
=\prod\limits_{m=1}^\infty \left(\left(1-\dfrac{b_1z_1}{m^2\pi^2}\right)^2\left(1-\dfrac{b_2z_2}{m^2\pi^2}\right)+\dfrac{b_3z^2_1z_2}{m^2\pi^2}\right)=0. \end{cases}
\end{equation}
Each functions in \eqref{f22} can be expanded into an infinite product of the functions from \eqref{f19}. The system \eqref{f22} has an infinite number of roots. The assumption $a_2\cdot b_2<0$ implies that no roots of \eqref{f22} lie in the coordinate planes.

Using \eqref{f21}, we find that
\begin{equation}\label{f23}
\sigma_{(1,1)}=\sum_{j=1}^\infty \frac {1}{z_{j1}}\cdot \frac {1}{z_{j2}}=\sum_{k,m=1}^{\infty} \frac {4a_2b_1}{\pi^4k^2m^2}-
\sum_{k,m=1}^{\infty} \frac{a_3b_2}{\pi^2(a_2m^2-b_2k^2)^2}.
\end{equation}
Using the formula \cite[No.~2, Section~5.1.25]{PBM86}
\[
\sum_{k=1}^{\infty} \frac{1}{(k^2+a^2)^2}=\frac{-1}{2a^4}+\frac{\pi}{4a^3}\coth(\pi a)+\frac{\pi^2}{4a^2}\cdot\frac 1{{\sinh}\,^2(\pi a)}
\]
we find the sum of the first series in the right hand side of \eqref{f23} 
\[
\sum_{k,m=1}^{\infty} \frac {4a_2b_1}{\pi^4k^2m^2}=\frac {a_2b_1}{9}
\]
and, respectively, the sum of the second series in the right hand side of \eqref{f23}
\begin{multline*}
\sum_{k,m=1}^{\infty} \frac {a_3b_2}{\pi^2(a_2m^2-b_2k^2)^2}=-\sum_{k=1}^{\infty} \frac {a_3}{2\pi^2 b_2k^4}
\\ -\sum_{k=1}^{\infty} \frac {a_3}{4\pi\sqrt{-a_2b_2}k^3} \coth(\pi \sqrt{-b_2/a_2}k)-
\sum_{k=1}^{\infty} \frac{a_3}{4a_2k^2} \cdot\frac 1{{\sinh}^2(\pi \sqrt{-b_2/a_2}k)}.
\end{multline*}
Here the sum of the first series in the right hand side is
\[
\sum_{k=1}^{\infty} \frac {a_3}{2\pi^2 b_2k^4}= \frac {a_3\pi^2}{180 b_2}.
\]

We now find the sum of the second series. Let $_2\Phi_1(e^{2t}, e^{2t}; e^{4t}, x)$ be a basic hypergeometric series (see, e.g., \cite[p.~793]{PBM86}). We use the known formula \cite[No.~13, Section~5.2.18]{PBM86}
\begin{multline*}
\sum_{k=1}^{\infty} \frac {x^{k-1}}{e^{2tk}-1}=\frac 1{x} \sum_{k=1}^{\infty} \frac {x^k}{e^{2tk}-1} \\
=\frac 1{x} \cdot \frac {x}{e^{4t}-1} {_2\Phi_1(e^{2t}, e^{2t}; e^{4t}, x)=\frac 1{e^{4t}-1}} {_2\Phi_1(e^{2t}, e^{2t}; e^{4t}, x)}.
\end{multline*}
Therefore,
\begin{multline}\label{f24}
\sum_{k=1}^{\infty} \frac {\coth(t k)}{k^3}=\sum_{k=1}^{\infty} \frac 1{k^3} + 2\sum_{k=1}^{\infty} \frac 1{k^3(e^{2ts}-1)} 
\\ =\zeta(3) + 2\frac 1{e^{4t}-1} \int_{0}^{1} \frac 1{y} \,dy \int_{0}^{y} \frac 1{v} \,dv \int_{0}^{v} {_2\Phi_1(e^{2t}, e^{2t}; e^{4t}, u)}\,du
\\ =\zeta(3) + \frac 1{e^{4t}-1} \int_{0}^{1} \ln^2{y} \cdot {_2\Phi_1(e^{2t}, e^{2t}; e^{4t}, y)}\,dy.
\end{multline}

In order to find the sum of the third series we rewrite $\dfrac 1{{\sinh}^2(tk)}$ as
\[
\frac 1{{\sin h}^2(tk)}= \left(\frac {2}{e^{tk}-e^{-tk}}\right)^2=\frac {4e^{2tk}}{(e^{2tk}-1)^2}  .
\]

Now, since
\[
\frac {\partial}{\partial t} \left[\frac {1}{e^{tk}-1} \right]=-\frac {2ke^{2tk}}{(e^{2tk}-1)^2},
\]
we have
\[
\frac{1}{(e^{2tk}-1)^2}=-\frac{1}{e^{2tk}-1}-\frac{1}{2k}\cdot\frac{\partial}{\partial t}\left[\frac {1}{e^{2tk}-1} \right].
\]
Consequently,
\[
\sum_{k=1}^{\infty} \frac 1{{\sinh}^2(tk)\cdot{k^2}}=-\frac{\partial}{\partial t}\left[\sum_{k=1}^\infty\frac 2{k^3(e^{2tk}-1)}\right]
\]

Thus, using the formula \eqref{f24}, we find that the series $\sigma_{(1,1)}$ is expressed in terms of the values of some integrals and known series without calculating the roots of the system.
\end{example}

\section*{Acknowledgements}

This research was supported by grants of the President of the Russian Federation for young scientists, project MD-197.2017.1 and for leading scientific schools, project NSh-9149.2016.1, by grant of the Government of the Russian Federation for investigations under the guidance of the leading scientists of the Siberian Federal University (contract No.~14.Y26.31.0006), and by the Russian Foundation for Basic Research, project 15-01-00277-a.

\end{document}